\documentclass{amsart}

\usepackage{amssymb}
\usepackage{latexsym}
\usepackage{amsmath}
\usepackage{euscript}
\usepackage{graphics}
\usepackage{todonotes}
\usepackage{scrextend}

      \def\dC{{\mathbb C}}
\def\dD{{\mathbb D}}

   \def\dN{{\mathbb N}}   
      \def\dR{{\mathbb R}}
   \def\dT{{\mathbb T}}   
      
   \def\dZ{{\mathbb Z}}

\def\bm\chi{\mbox{\boldmath$\chi$}}

\def\diag{{\rm diag\,}}

\let\xker=\ker \def\ker{{\xker\,}}

\def\sg{\operatorname{sign}}

\def\sg{\operatorname{sign}}

\def\deg{\operatorname{deg}}

\newtheorem{theorem}{Theorem}[section]
\newtheorem{proposition}[theorem]{Proposition}
\newtheorem{corollary}[theorem]{Corollary}
\newtheorem{lemma}[theorem]{Lemma}
\newtheorem{definition}[theorem]{Definition}
\theoremstyle{remark}

\newtheorem{remark}[theorem]{Remark}

\numberwithin{equation}{section}

\DeclareMathOperator{\Rl}{Re}
\newcommand{\nri}{n\rightarrow\infty}
\newcommand{\eitheta}{e^{i\theta}}

\date{\today}

\author[M. Derevyagin]{Maxim Derevyagin}
\address{
Maxim Derevyagin\\
University of Mississippi\\
Department of Mathematics\\
Hume Hall 305 \\ 
P. O. Box 1848 \\
University, MS 38677-1848, USA }
\email{derevyagin.m@gmail.com}

\author[B. Simanek]{Brian Simanek}

\address{Brian Simanek\\
Baylor University\\
Department of Mathematics\\
One Bear Place \#97328\\
Waco, TX 76798-7328, USA} 
\email{Brian\_Simanek@baylor.edu}

\subjclass{Primary 42C05; Secondary 30D30, 46C20.}
\keywords{Orthogonal polynomials on the unit circle, pseudo-Carath\'eodry function, Szeg\H{o} function, Szeg\H{o} mapping, asymptotic, non-symmetric Jacobi matrix}

\begin{document}

\title[Asymptotics for orthogonal polynomials]{Asymptotics for polynomials orthogonal in an indefinite metric}

\begin{abstract}
We continue studying polynomials generated by the Szeg\H{o} recursion when a finite number of Verblunsky coefficients lie outside the closed unit disk. We prove some asymptotic results for the corresponding orthogonal polynomials and then translate them to the real line to obtain the Szeg\H{o} asymptotics for the resulting polynomials. The latter polynomials give rise to a non-symmetric tridiagonal matrix but it is a finite-rank perturbation of a symmetric Jacobi matrix.  
\end{abstract}

\maketitle

\section{Introduction}

In the theory of orthogonal polynomials, two distinguished classes have historically received special attention, namely that in which the measure of orthogonality is supported on the unit circle (OPUC) and that in which the measure of orthogonality is supported on the real line (OPRL).  The prominent distinguishing feature of these classes is the existence of a recursion relation satisfied by the orthogonal polynomials.  In the setting of OPUC, this gives rise to the sequence of so-called \textit{Verblunsky coefficients}, which we denote by $\{\alpha_n\}_{n=0}^{\infty}$, each of which is a complex number in the open unit disk.  Verblunsky's Theorem establishes a bijection between such sequences and infinitely supported probability measures on the unit circle (see \cite[Chapter 1]{OPUC1}).  Similarly, Favard's Theorem establishes a bijection between pairs of bounded real sequences $\{a_n\}_{n\in\dN}$ and $\{b_n\}_{n\in\dN}$ where each $a_n>0$ and probability measures with infinite and compact support on the real line (see \cite[Theorem 1.3.7]{Rice}).  A common theme of the research in these fields has been to investigate the relationship between the measure of orthogonality and the corresponding sequence or sequences (see \cite{OPUC1,OPUC2,Rice}).  We will refer to any sequence or sequences to which we can apply Verblunsky's Theorem or Favard's Theorem as belonging to the \textit{classical case}.

In \cite{DS17}, a special non-classical class of Verblunsky coefficients was studied and unlike the classical case, that class does not correspond to meausres on the unit circle $\dT$. In this note we proceed with the exploration of this class and we will use the notation from \cite{DS17}, which is in turn inherited from \cite{OPUC1,OPUC2}. More precisely, we consider sequences $\{\alpha_n\}_{n=0}^{\infty}$ of complex numbers for which there exists a natural number $N$ such that
\begin{equation}\label{IndCond}
\begin{split}
|\alpha_n|&\ne 1, \quad n=0,1,2, \dots, N-1,\\
|\alpha_n|&<1, \quad n=N, N+1, N+2, \dots.
\end{split}
\end{equation}
Then, for any nonnegative integer $n$ it is still possible to define a monic polynomial $\Phi_{n+1}$ of degree $n+1$ by iterating the Szeg\H{o} recurrence
\begin{equation}\label{SzRec}
\begin{split}
\Phi_{n+1}(z)=z\Phi_n(z)-\overline{\alpha}_n\Phi_n^*(z)\\
\Phi_{n+1}^*(z)=\Phi_n^*(z)-\alpha_nz\Phi_n(z),
\end{split}
\end{equation}
provided that we set the initial condition to be
\begin{equation}\label{InCond}
\Phi_0(z)=1
\end{equation}
and $\Phi_n^*$ is the polynomial reversed to $\Phi_n$, that is,
\begin{equation}\label{RevPol}
\Phi_n^*(z)=z^n\overline{\Phi_n(1/\overline{z})}.
\end{equation}

Given a sequence $\{\alpha_n\}_{n=0}^{\infty}$, we will often refer to the $m$-times stripped sequence given by $\{\alpha_n\}_{n=m}^{\infty}$.  If a sequence satisfies \eqref{IndCond}, then the $N$-times stripped sequence satisfies the hypotheses of Verblunsky's theorem and thus corresponds to a measure on the unit circle.  Associated to such a measure is the sequence $\{f_n\}_{n=N}^{\infty}$ of Schur functions that satisfy the recursive relation
%Also, in our case the sequence of Verblunsky coefficients generates a sequence of functions $\{f_n\}_{n=0}^{\infty}$ exactly as the classical case and these functions obey the relation
\begin{equation}\label{schurform}
f_n(z) = \frac{\alpha_n + zf_{n+1}(z)}{1+\bar\alpha_n zf_{n+1}(z)}, \quad n=N, N+1, N+2, \dots. 
\end{equation}
This recursion can be iterated with any choice of $\{\alpha_n\}_{n=0}^{N-1}$ and hence to any sequence $\{\alpha_n\}_{n=0}^{\infty}$ satisfying \eqref{IndCond} we can associate a sequence of functions $\{f_n\}_{n=0}^{\infty}$ that obeys \eqref{schurform}.  Since we allow some Verblunsky coefficients to be outside the closed unit disk, in this situation one cannot say that all $f_n$'s are Schur functions. However, \eqref{IndCond} ensures that $f_n$ is a Schur function for $n=N, N+1, \dots$.  Once we have the sequence of functions $\{f_n\}_{n=0}^{\infty}$, it is natural to define a function that will play the role of a Carth\'eodory function in our theory. To this end, let us set $f=f_0$ and use the standard formula to define $F$ as in \cite{DS17}
\begin{equation}\label{Fdef}  
F(z) := \frac{1+zf(z)}{1-zf(z)}.
\end{equation}
As is shown in \cite[Proposition 2.1]{DS17}, Khrushchev's formula still holds for this $F$, that is, 
\begin{equation}\label{Khrushev}
\Rl F(z)=\omega_{n-1}\frac{(1-|f_n(z)|^2)}{|\Phi_n^*(z)-z\Phi_n(z)f_n(z)|^2}, \quad n\in\dN,
\end{equation}
for Lebesgue almost every $z\in\dT$, where 
\[
\omega_{n}:=\prod_{j=0}^{n}(1-|\alpha_j|^2).
\]
It will be sometimes convenient to have a separate notation for the sign of $\omega_{n}$. So,  let $\epsilon_{n}$ be the sign of $\omega_{n}$, that is, $\epsilon_{n}=\sg\omega_{n}$.  With this notation we can say that $\epsilon_{N-1} F$ is a pseudo-Carath\'eodory function, which was proved in \cite[Theorem 2.3]{DS17}.  Such functions are related to classical Carath\'eodory functions in that pseudo-Carath\'eodory functions of finite index admit the representation
\[
F(z)=v(z)\overline{v(1/\bar{z})}g(z),
\]
where $g$ is a classical Carath\'eodory function with $|g(0)|=1$ and $v(z)$ is an outer rational function (see  \cite[Theorem 3.1]{DGK86} and \cite{ADL07}). 

In \cite{DS17} an analog of Szeg\H{o}'s Theorem was proven in this non-classical setting.  The aim of the present paper is to further develop the asymptotic theory of polynomials generated from the Szeg\H{o} recursion using a sequence satisfying \eqref{IndCond}.  We will pay special attention to adapting the relationship between OPUC and OPRL to this non-classical setting.  In the next section we define the generalized Szeg\H{o} function and then find the asymptotics for the normalized polynomials $\{\varphi_n^*\}_{n=0}^{\infty}$. In Section 3 we re-examine the Szeg\H{o} mapping and then use it to translate results to the real line. Finally, the last section provides some explicit examples.

For the remainder of this paper, we will always assume that $\{\alpha_n\}_{n=0}^{\infty}$ is a sequence satisfying \eqref{IndCond} and $\{\Phi_n\}_{n=0}^{\infty}$ is the corresponding sequence of monic polynomials generated from this sequence using the Szeg\H{o} recursion.  The corresponsing sequence $\{f_n\}_{n=0}^{\infty}$ and function $F$ will be defined as above.

\section{The Generalized Szeg\H{o} Function}

In this section we will extend the notion of the Szeg\H{o} function to the non-classical setting and use it to prove some asymptotic results for the polynomials $\{\Phi_n\}_{n=0}^{\infty}$.  We begin by adapting \cite[Theorem 8.48]{Kh08} to this setting, which is a result about orthonormal polynomials, so let us define
\[
\varphi_n(z):=\Phi_n(z)/\sqrt{|\omega_{n-1}|}.
\] 
\begin{theorem}\label{L1con}
If $\log\Rl F\in L^1(\dT)$ then
\begin{equation}
\lim_{n\to\infty}\int_{\dT}\left|\log\frac{\epsilon_{n-1}}{|\varphi_n^*(e^{i\theta})|^2}-\log\Rl F(e^{i\theta})\right| \, \frac{d\theta}{2\pi}=0.
\end{equation}
\end{theorem}

\begin{proof}
From \eqref{Khrushev} we see that
\[
\log(\epsilon_{n-1}|\varphi_n^*(e^{i\theta})|^2\Rl F(e^{i\theta}))=
\log(1-|f_n(e^{i\theta})|^2)-
2\log\left|1-e^{i\theta}\frac{\varphi_n(e^{i\theta})}{\varphi_n^*(e^{i\theta})}f_n(e^{i\theta})\right|,
\]
which implies that
\[
\begin{split}
&\left|\log\frac{\epsilon_{n-1}}{|\varphi_n^*(e^{i\theta})|^2}-\log\Rl F(e^{i\theta})\right|\le
|\log(1-|f_n(e^{i\theta})|^2)|\\
&\qquad\qquad\qquad+2\log^+\left|1-e^{i\theta}\frac{\varphi_n(e^{i\theta})}{\varphi_n^*(e^{i\theta})}f_n(e^{i\theta})\right|+
2\log^-\left|1-e^{i\theta}\frac{\varphi_n(e^{i\theta})}{\varphi_n^*(e^{i\theta})}f_n(e^{i\theta})\right|,
\end{split}
\]
where we use the standard notation $g^+(x):=\max(g(x),0)$ and $g^-(x):=g^+(x)-g(x)$.  Next, we notice that $f_n$ is a Schur function for all natural numbers $n\ge N$ and, therefore, by Boyd's theorem we get
\[
\lim_{n\to\infty}\int_{\dT}|\log(1-|f_n(e^{i\theta})|^2)|\, \frac{d\theta}{2\pi}=
\lim_{n\to\infty}\int_{\dT}\left(-\log(1-|f_n(e^{i\theta})|^2)\right)\, \frac{d\theta}{2\pi}=0.
\]
Since it follows from the reasoning given in the proof of \cite[Theorem 3.4]{DS17} that 
\[
\int_{\dT}\log\left|1-e^{i\theta}\frac{\varphi_n(e^{i\theta})}{\varphi_n^*(e^{i\theta})}f_n(e^{i\theta})\right|\, \frac{d\theta}{2\pi}
\]
converges to $0$ as $n$ tends to $\infty$, we need to only show that 
\[
\lim_{n\to\infty}\int_{\dT}\log^+\left|1-e^{i\theta}\frac{\varphi_n(e^{i\theta})}{\varphi_n^*(e^{i\theta})}f_n(e^{i\theta})\right|\, \frac{d\theta}{2\pi}=0.
\]
The latter is also a consequence of Boyd's theorem and the elementary facts
$\log^+|1-z|\le\log(1+|z|)$ whenever $0\le |z|\le 1$ and $\log(1+x)\le x \le \log\frac{1}{1-x}$ for $x\in[0,1]$. Indeed, we have the following chain of inequalities:
\[
\begin{split}
\int_{\dT}\log^+\left|1-e^{i\theta}\frac{\varphi_n(e^{i\theta})}{\varphi_n^*(e^{i\theta})}f_n(e^{i\theta})\right|\, \frac{d\theta}{2\pi}&\le \int_{\dT}|f_n(e^{i\theta})|\, \frac{d\theta}{2\pi}\le
\left(\int_{\dT}|f_n(e^{i\theta})|^2\, \frac{d\theta}{2\pi}\right)^{\frac{1}{2}}\\
&\le \left(\int_{\dT}\log\frac{1}{1-|f_n(e^{i\theta})|^2}\, \frac{d\theta}{2\pi}\right)^{\frac{1}{2}}.
\end{split}
\] 
Hence we arrive at the desired result.
\end{proof}

The conclusion of Theorem \ref{L1con} is a statement about strong convergence in $L^1$, which, as is known, implies weak convergence. In particular, the following statement holds true.

\begin{corollary}
If $\log\Rl F\in L^1(\dT)$ then
\begin{equation}\label{WeakConv}
\lim_{n\to\infty}\int_{\dT}e^{i k\theta}\log\frac{\epsilon_{n-1}}{|\varphi_n^*(e^{i\theta})|^2} \, \frac{d\theta}{2\pi}
=\int_{\dT}e^{i k\theta}\log\Rl F(e^{i\theta})\, \frac{d\theta}{2\pi},\quad k\in\dZ.
\end{equation}
\end{corollary}

In trying to adapt the classical theory to our setting we introduce an analog of the Szeg\H{o} function as well.

\begin{definition}
Suppose that $\log\Rl F\in L^1(\dT)$.  The Szeg\H{o} function, $D(z)$, is defined by
\[
D(z)=\exp\left(\frac{1}{4\pi}\int_{\dT}\frac{e^{i\theta}+z}{e^{i\theta}-z}\log\left(\epsilon_{N-1}\Rl F(e^{i\theta})\right)\, {d\theta}\right).
\]
\end{definition}

Let us begin our analysis by deriving an alternate formula for $D(z)$.  The calculations and formulas in \cite{DS17} show us that \cite[Lemma 2.9.2]{OPUC1} (and its proof) remains valid in this non-classical setting.  To be precise, let $F_j$ be the Carath\'eodory function for the $j$-times stripped sequence $\{\alpha_n\}_{n=j}^{\infty}$ and let $f=f_0$.  Then
\[
\frac{\Rl F(z)}{\Rl F_1(z)}=\frac{|1-\bar{\alpha}_0f|^2|1-zf_1|^2(1-|zf|^2)}{(1-|\alpha_0|^2)|1-zf|^2(1-|f|^2)}.
\]
Therefore, we define the relative Szeg\H{o} function by
\[
(\delta_0D)(z)=\frac{(1-\bar{\alpha}_0f(z))(1-zf_1(z))}{|\rho_0|(1-zf(z))},
\]
where $\rho_n=\sqrt{1-|\alpha_n|^2}$.  Notice that $|\rho_0|$ is well-defined even if $|\alpha_0|>1$.  Then the above formula shows that for all $z\in\dT$ it holds that
\[
\Rl F(z)=\frac{|\rho_0|^2}{\rho_0^2}\Rl F_1(z)|(\delta_0D)(z)|^2,
\]
which is an analog of \cite[Equation 2.9.17]{OPUC1}.  Similarly, if we define
\[
(\delta_jD)(z)=\frac{(1-\bar{\alpha}_jf_j(z))(1-zf_{j+1}(z))}{|\rho_j|(1-zf_j(z))},
\]
then we iterate the above reasoning and arrive at
\begin{align*}
\Rl F(z)&=\frac{|\rho_0\rho_1\cdots\rho_{N-1}|^2}{(\rho_0\rho_1\cdots\rho_{N-1})^2}\Rl F_N(z)\left|\prod_{j=0}^{N-1}(\delta_jD)(z)\right|^2=\epsilon_{N-1}\Rl F_N(z)\left|\prod_{j=0}^{N-1}(\delta_jD)(z)\right|^2\\
&=\epsilon_{N-1}\Rl F_N(z)\left|\frac{1-zf_N(z)}{1-zf(z)}\prod_{j=0}^{N-1}\frac{1-\bar{\alpha}_jf_j(z)}{|\rho_j|}\right|^2
\end{align*}
We need the following lemma.

\begin{lemma}\label{l1f}
For all $j\geq0$, we have $\log|1-\bar{\alpha}_jf_j|\in L^1(\dT)$.
\end{lemma}

\begin{proof}
By \cite[Equation 2.2]{DS17}, we know that there are polynomials $R_j(z)$ and $Q_j(z)$ such that
\[
f_j(z)=\frac{R_j(z)+zQ_j^*(z)f_N(z)}{Q_j(z)+zR_j^*(z)f_N(z)}.
\]
Furthermore, we know that $Q_j(0)=1$ and $R_j(0)=\alpha_j$.  We then find
\begin{align*}
1-\bar{\alpha}_jf_j(z)=\frac{Q_j(z)+zR_j^*(z)f_N(z)-\bar{\alpha}_jR_j(z)-\bar{\alpha}_jzQ_j^*(z)f_N(z)}{Q_j(z)+zR_j^*(z)f_N(z)}.
\end{align*}
Both the numerator and denominator of this fraction are clearly in $H^{\infty}(\dD)$.  Furthermore, neither the numerator not the denominator are identically zero as can be seen by evaluation at zero and using the fact that $|\alpha_j|\neq1$.  The desired result now follows from \cite[Theorem 17.17]{Rudin}.
\end{proof}

Recall that $F_N$ is a classical Carath\'eodory function.  The definition of $D(z)$ and Lemma \ref{l1f} lead us to the following formula
\begin{align*}
D(z)&=\frac{D_N(z)}{|\rho_0\rho_1\cdots\rho_{N-1}|}\exp\left(\frac{1}{2\pi}\int_{\dT}\frac{e^{i\theta}+z}{e^{i\theta}-z}\log\left|\frac{1-e^{i\theta}f_N(e^{i\theta})}{1-e^{i\theta}f(e^{i\theta})}\right|d\theta\right)\\
&\qquad\qquad\qquad\times\prod_{j=0}^{N-1}\exp\left(\frac{1}{2\pi}\int_{\dT}\frac{e^{i\theta}+z}{e^{i\theta}-z}\log|1-\bar{\alpha}_jf_j(e^{i\theta})|d\theta\right),
\end{align*}
where $D_N$ is the Szeg\H{o} function for the $N$-times stripped sequence $\{\alpha_n\}_{n=N}^{\infty}$.  We can then apply \cite[Equation 2.9.12]{OPUC1} to $D_N$ to derive the following formula for $D(z)$ in terms of the Schur iterates of $f$:

\begin{theorem}
If $|z|<1$ and $\log\Rl F\in L^1(\dT)$, then
\begin{align*}
D(z)&=\frac{\prod_{j=N}^{\infty}(1-\bar{\alpha}_jf_j(z))}{\prod_{j=0}^{\infty}|\rho_j|}\exp\left(\frac{1}{2\pi}\int_{\dT}\frac{e^{i\theta}+z}{e^{i\theta}-z}\log\left|\frac{\prod_{j=0}^{N-1}(1-\bar{\alpha}_jf_j(e^{i\theta}))}{1-e^{i\theta}f(e^{i\theta})}\right|d\theta\right)\\
%&\qquad\qquad\qquad\qquad\qquad\qquad\times\exp\left(\frac{-1}{2\pi}\int_{\dT}\frac{e^{i\theta}+z}{e^{i\theta}-z}\log\left|1-e^{i\theta}f(e^{i\theta})\right|d\theta\right)
\end{align*}
\end{theorem}

\begin{proof}
Starting with our above formula and an application of \cite[Equation 2.9.12]{OPUC1}, all that remains to prove is that
\[
1-zf_N(z)=\exp\left(\frac{1}{2\pi}\int_0^{2\pi}\frac{\eitheta+z}{\eitheta-z}\log|1-\eitheta f_N(\eitheta)|d\theta\right),\qquad z\in\dD.
\]
This equality follows from \cite[Theorem 17.17]{Rudin} and the fact that $1-zf_N(z)$ is an outer function, which was proven in \cite[Lemma 2.7.8]{OPUC1}.
\end{proof}

To formulate the next result, which is a generalization of \cite[Theorem 2.4.1, part (iv)]{OPUC1}, we should recall from \cite{DS17} that the zeros of $\varphi_n^*$ inside $\dD$ will either tend to the poles of $F$ in $\dD$ or to the boundary of the unit disk as $\nri$. So, let $B_n$ be the Blaschke product formed by the zeroes of $\varphi_n^*$ inside $\dD$, that is
\[
B_{n}(z)=\prod_{j=1}^{k}\frac{|\lambda_{j,n}|}{\lambda_{j,n}}\frac{\lambda_{j,n}-z}{1-\bar{\lambda}_{j,n}z},
\]
where $\lambda_{j,n}\in\dD$ and $k$ is independent of $n$ for sufficiently large $n$.
Then $B_n$ converges to a function $B$ locally uniformly in $\dD$, where $B$ is the Blaschke product constructed from the limit points of the zeroes of $\varphi_n^*$ inside $\dD$.  With this notation, we can now state the main result of this section.

\begin{theorem}\label{normalszego}
If $\log\Rl F\in L^1(\dT)$ then
\begin{equation}\label{PhiStarConv}
\lim_{n\to\infty}\varphi_n^*(z)=B(z)D^{-1}(z)
\end{equation}
uniformly on compact subsets of $\dD$.
\end{theorem} 
\begin{proof}
Let us observe that $\varphi_n^*(z)/B_n(z)$ is an outer function by \cite[Corollary 4.7, Chapter II]{G81}. Hence, 
$\varphi_n^*(z)$ does not have an inner part. So, by the theorem from \cite[Chapter IV, Section D.4]{K98} we have
\[
\begin{split}
\varphi_n^*(z)&=B_n(z)\exp\left(\frac{1}{2\pi}\int_{\dT}\frac{e^{i\theta}+z}{e^{i\theta}-z}\log|\varphi_n^*(e^{i\theta})|\,{d\theta}\right)\\
&=B_n(z)\exp\left(-\frac{1}{4\pi}\int_{\dT}\frac{e^{i\theta}+z}{e^{i\theta}-z}\log|\varphi_n^*(e^{i\theta})|^{-2}\,{d\theta}\right)
\end{split}
\] 
It remains to notice that $\epsilon_n=\epsilon_{N-1}$ for all $n\ge N-1$ and to use \eqref{WeakConv} together with the criterion of convergence of analytic functions in terms of derivatives of all orders at $0$. 
\end{proof}

Evaluation of \eqref{PhiStarConv} at $0$ gives Szeg\H{o}'s theorem from \cite{DS17}.

\begin{corollary}
If $\log\Rl F\in L^1(\dT)$, then it holds that
\begin{equation}\label{StarRatio}
\lim_{n\to\infty}\frac{\Phi_{n+1}^*(z)}{\Phi_{n}^*(z)}=1
\end{equation}
locally uniformly on $\overline{\dD}\setminus\{\lambda_j\}_{j=1}^{k}$, where $\{\lambda_j\}_{j=1}^{k}$ is the set of zeroes of $B$.
\end{corollary}

\begin{proof}
For $z\in\dD\setminus\{\lambda_j\}_{j=1}^{k}$, this immediately follows form \eqref{PhiStarConv}. As for the boundary of the unit disk, see \cite[formula (1.7.12)]{OPUC1}.
\end{proof}

\section{The Szeg\H{o} mapping and Geronimus relations}

In this section we will explore some connections between our results and the theory of orthogonal polynomials on the real line.  Our main goal is to extend some of the results from \cite[Section 13.1]{OPUC2} to our setting and then apply them to obtain an analog of \cite[Theorem 13.3.2]{OPUC2}.

Mimicking the classical case of the Szeg\H{o} mapping, we start with a sequence of real Verblunsky coefficients subject to \eqref{IndCond}. Then, we get a sequence of polynomials $\{\Phi_n\}_{n=0}^{\infty}$ from the Szeg\H{o} recurrence  \eqref{SzRec} and, since each $\alpha_n\in\dR$, each $\Phi_n$ has real coefficients.  In the classical setting of orthogonal polynomials on the unit circle, this situation allows one to obtain a corresponding measure on the interval $[-2,2]$ by means of the Szeg\H{o} mapping (see \cite[Chapter 13]{OPUC2}).  One can then explore asymptotics of the corresponding orthogonal polynomials in terms of objects from the unit circle setting.  Although we are not dealing with measures on the unit circle, there is still a portion of the theory that can be carried over to our setting.  We begin with the following observation.

\begin{proposition}\label{subex}
Let $\{\alpha_n\}_{n=0}^{\infty}$ be a sequence of Verblunsky coefficients satisfying \eqref{IndCond} and let $\{\Phi_n\}_{n=0}^{\infty}$ be the corresponding sequence of monic orthogonal polynomials.  If
\[
\Phi_n(z)=z^n+\sum_{j=0}^{n-1}t_jz^j,
\]
then define a linear functional $\mu$ on the space of Laurent polynomials by $\mu(1)=1$ and
\begin{align}
\nonumber\mu(z^n)&=-\sum_{j=0}^{n-1}t_j\mu(z^j),\qquad n\in\dN,\\
\label{symmetry}\mu(z^{-n})&=\overline{\mu(z^n)},\qquad n\in\dN.
\end{align}
Then
\[
\limsup_{n\to\infty}|\mu(z^n)|^{1/n}<\infty.
\]
\end{proposition}

\begin{proof}
First consider the case $|\alpha_n|<1$ for all $n\geq0$ so that $\mu$ is given by integration against a probability measure on the unit circle.  By \cite[Theorem 1.5.5]{OPUC1} we know that
\[
\mu(z^{-n})=P_n(\alpha_0,\ldots,\alpha_{n-1},\bar{\alpha}_0,\ldots,\bar{\alpha}_{n-1})
\]
for some polynomial $P_n$ in $2n$ variables.  Furthermore, the function
\begin{equation}\label{carform}
1+2\sum_{n=1}^{\infty}\mu(z^{-n})z^n
\end{equation}
is the Carath\'eodory function for the measure $\mu$.  In the general case, it follows from the construction that the quantities $\{F^{(n)}(0)\}_{n=0}^{\infty}$ are each continuous functions of the Verblunsky coefficients and in fact
\[
\frac{F^{(n)}(0)}{n!}=2P_n(\alpha_0,\ldots,\alpha_{n-1},\bar{\alpha}_0,\ldots,\bar{\alpha}_{n-1}),\qquad n\in\dN.
\]
Therefore, the function \eqref{carform} is Maclaurin series for $F$.  Since $F$ does not have a singularity at $0$, the desired claim follows.
\end{proof}

The distribution $\mu$ defined in Proposition \ref{subex} can be extended by linearity to Maclaurin series with non-zero radius of convergence.  Following ideas from \cite{Ger40,JNT,Zayed}, we set $\mu(z^j\bar{z}^k)=\mu(z^{j-k})$ so the polynomials $\{\Phi_n\}_{n=0}^{\infty}$ are orthogonal with respect to $\mu$, i.e.
\[
\left\langle\bar{w}^k\Phi_n(w),\mu\right\rangle_w=0, \quad k=0,1,\dots, n-1,
\]
and 
\[
\left\langle\bar{w}^n\Phi_n(w),\mu\right\rangle_w\ne 0,
\]
Consequently, we also have that 
\[
\left\langle\bar{w}^k\Phi^*_n(w),\mu\right\rangle_w=0, \quad k=1,2,\dots, n.
\]
We should mention that the distribution $\mu$ need not be unique, but its action on Laurent polynomials and Maclaurin series is uniquely determined by the sequence $\{\Phi_n\}_{n=0}^{\infty}$, the symmetry relation (\ref{symmetry}), and the normalization $\langle1,\mu\rangle_w=1$.  Since we will only be applying $\mu$ to Laurent polynomials and Maclaurin series, we will refer to this properly normalized $\mu$ as \textit{the} distribution corresponding to the sequence $\{\alpha_n\}_{n=0}^{\infty}$.

The next step is to transform the distribution $\mu$ by analogy with the classical case to obtain a new distribution that can be applied to functions defined on the interval $[-2,2]$.  This transformed distribution $\gamma$ is defined via the formula
\begin{equation}\label{SzIm}
\left\langle g(x),\gamma\right\rangle_x=\left\langle g\left(w+\bar{w}\right),\mu\right\rangle_w,
\end{equation}
where $x=w+\bar{w}$ and $g$ is any test function which can, for instance, be thought of as an arbitrary polynomial.
We can now formulate a result adapting the Szeg\H{o} mapping to the case in question.

\begin{theorem}[Szeg\H{o} Mapping]
Let $\{\alpha_n\}_{n=0}^{\infty}$ be a sequence of real Verblunsky coefficients obeying \eqref{IndCond} and let 
$\mu$ be the corresponding distribution. Then the distribution $\gamma$ defined by \eqref{SzIm} is quasi-definite, meaning there exists a sequence of monic polynomials $\{P_n\}_{n=0}^{\infty}$ satisfying $\deg(P_n)=n$ and such that
\[
\left\langle x^kP_n(x),\gamma\right\rangle_x=0, \quad k=0,1,\dots, n-1
\]
and 
\[
\left\langle x^nP_n(x),\gamma\right\rangle_x\ne 0
\]
for $n=0,1,\dots$. In addition, the polynomials $P_n$ obey the formula
\begin{equation}\label{GerPol}
P_n\left(z+\frac{1}{z}\right)=(1-\alpha_{2n-1})^{-1}z^{-n}\left[\Phi_{2n}(z)+\Phi_{2n}^*(z)\right],
\end{equation} 
where we set $\alpha_{-1}=-1$.
\end{theorem}

\begin{proof}
We have already observed that when the Verblunsky coefficients are real, the polynomial $\Phi_n$ has real coefficients. Therefore, \cite[Lemma 13.1.4]{OPUC2} still ensures that the formula \eqref{GerPol}  correctly defines monic polynomials $P_n(x)$ of degree $n$ for $n=0,1,\dots$. The rest of the proof proceeds exactly as the proof of \cite[Theorem 13.1.5]{OPUC2} with integration replaced by pairing with a distribution.
\end{proof}

It is well known that polynomials orthogonal with respect to a quasi-definite distribution satisfy three-term recurrence relations (see \cite[Theorem 4.1]{Chihara}). In other words, there exists two sequences $\{b_n\}_{n=1}^{\infty}$ and $\{c_n\}_{n=1}^{\infty}$ such that
\begin{equation}\label{3term}
xP_n(x)=P_{n+1}(x)+b_{n+1}P_n(x)+c_nP_{n-1}(x), \quad n=0,1,\dots
\end{equation}
with the convention $P_{-1}=0$.  If the distribution $\gamma$ is a positive measure, then one can further state that $c_n\geq0$ for all $n\in\dN$, but in general this is not the case.  The Geronimus relations \cite[Theorem 13.1.7]{OPUC2} provide a relationship (in the classical case) between the Verblunsky coefficients and the recursion coefficients $\{b_n\}_{n=1}^{\infty}$ and $\{c_n\}_{n=1}^{\infty}$.  Our next result extends this relation to the case considered in this paper.

\begin{theorem}[Geronimus relations]
For the recurrence coefficients from \eqref{3term} we have 
\begin{eqnarray}
\label{GR1} c_{n+1}&=&(1-\alpha_{2n-1})(1-\alpha_{2n}^2)(1+\alpha_{2n+1})\\
\label{GR2} b_{n+1}&=&(1-\alpha_{2n-1})\alpha_{2n}-(1+\alpha_{2n+1})\alpha_{2n-2}
\end{eqnarray}
for $n=0,1,\dots$, where we set $\alpha_{-1}=0$.
\end{theorem}

\begin{proof}
The proof of \eqref{GR2} is exactly the same as the proof of the corresponding relation in \cite[Theorem 13.1.7]{OPUC2}. 
The proof of \eqref{GR1} is also similar to the proof of \cite[Theorem 13.1.7]{OPUC2} but with a few details changed. So, we begin by multiplying \eqref{3term} by $P_{n-1}$ and applying the distribution $\gamma$, which yields
\[
c_n=\frac{\left\langle xP_{n-1}(x)P_{n}(x),\gamma\right\rangle_x}{\left\langle P_{n-1}^2(x),\gamma\right\rangle_x}.
\]
This reduces to
\[
c_n=\frac{\left\langle P_{n}^2(x),\gamma\right\rangle_x}{\left\langle P_{n-1}^2(x),\gamma\right\rangle_x}
\]
and then we can proceed exactly as in the proof of \cite[Theorem 13.1.7]{OPUC2} but replacing integration by pairing with a distribution.
\end{proof}

As a result of the Geronimus relations, we see a striking difference between the classical and non-classical cases. Namely, the fact that $|\alpha_n|>1$ for some values of $n$ will cause $c_{m}<0$ for some values of $m$, which never happens when $\gamma$ is a positive measure.  However, the condition (\ref{IndCond}) implies that $c_m<0$ for only a finite number of values of $m\in\dN$.  If we define
\[
p_n:=\frac{P_n}{\sqrt{|c_n\cdots c_1|}}
\]
then these polynomials satisfy the three-term recurrence relation
\[
xp_n(x)=\sqrt{|c_{n+1}|}p_{n+1}(x)+b_{n+1}p_n(x)+\frac{c_n}{\sqrt{|c_n|}}p_{n-1}(x).
\]
This recurrence relation can be expressed in terms of the following tri-diagonal matrix:
\[
H=\begin{pmatrix}
  b_1 & \sqrt{|c_1|}&  &\\
  \frac{c_1}{\sqrt{|c_1|}}& b_2& \sqrt{|c_2|}&\\
     & \frac{c_2}{\sqrt{|c_2|}}&b_3&\ddots\\
     &&\ddots&\ddots&    
      \end{pmatrix},      
\]
where all entries away from the three main diagonals are zero.  Recall that the assumption \eqref{IndCond} only allows for finitely many coefficients $c_m$ to be negative.  Consequently, the matrix $H$, which is obviously not symmetric, is a finite rank perturbation of a symmetric Jacobi matrix. This type of matrix is a very particular case of the generalized Jacobi matrices introduced and studied in \cite{DD04,DD07}.  The generalized Jacobi matrices play the same role for indefinite Hamburger moment problems as classical Jacobi matrices play for Hamburger moment problems. In other words, the generalized Jacobi matrices naturally appear as an operator model for the step-by-step solution of the Hamburger moment problems in the class of generalized Nevanlinna functions and the latter can be thought of as rational perturbations of classical Nevanlinna functions (see \cite{DD04,DD07} for further details). %In particular, the bounded generalized Jacobi matrices are in one-to-one correspondence with the class of rational perturbations of Markov functions.

The matrix $H$ also gives rise to a self-adjoint operator in a Pontryagin space. To see this, define the sequence $\{\Delta_n\}_{n=1}^{\infty}$ so that $\Delta_1=1$, each $\Delta_n=\pm1$, and
\[
\frac{c_n}{\sqrt{|c_n|}}=\Delta_n\Delta_{n+1}\sqrt{|c_n|}.
\]
Note that since $c_m<0$ for only a finite number of values of $m\in\dN$, $\Delta_n$ has the same sign for all sufficiently large $n$. Then define the diagonal matrix $G=\diag(\Delta_1,\Delta_2,\Delta_3, \dots)$ and consider the bilinear form on $\ell^2$ given by
\[
(x,y)_G=\langle Gx,y\rangle_{\ell^2}, \quad x,y\in \ell^2.
\]  
Notice that $GH$ is a symmetric matrix so that
\[
(Hx,y)_G=\langle GHx,y\rangle_{\ell^2}=\langle x,GHy\rangle_{\ell^2}=\langle Gx,Hy\rangle_{\ell^2}=(x,Hy)_G,
\]
so $H$ is self-adjoint as an operator on this space.

Following \cite{DD04,DD07} we can also introduce the $m$-function of $H$ via the formula
\[
m(z)=((H-z)^{-1}e,e)_G, \quad e=(1,0,0,\dots)^{\top}.
\]
%This $m$-function is a generalized Nevanlinna function.
The function $m$ is analytic on its domain, which is the resolvent set of the operator $H$.  It could have poles in the upper half-plane, but since $H$ is a bounded operator, it is certainly defined in some neighborhood of infinity. This can also be seen from the theory we have already developed and the following result.

\begin{theorem}
Let $\{\alpha_n\}_{n=0}^{\infty}$ be a sequence of real Verblunsky coefficients obeying \eqref{IndCond} and let 
$\mu$ and $\gamma$ be the corresponding distributions related by means of the Szeg\H{o} mapping. The function $F$ given by \eqref{Fdef} admits the following representation in some neighborhood of zero:
\[
F(z)=\left\langle\frac{w+z}{w-z},\mu\right\rangle_w
\]
and the corresponding $m$-function obeys the formula
\[
m(z)=\left\langle\frac{1}{x-z},\gamma\right\rangle_x
\]
for all $z$ in some neighborhood of infinity.  Moreover, the $m$-function and $F$ are related by
\begin{equation}\label{CN_r}
F(z)=(z-z^{-1})m(z+z^{-1}),\qquad z\in\dD\setminus\{\lambda_j\}_{j=1}^{k},
\end{equation}
where $\{\lambda_j\}_{j=1}^{k}$ is the set of zeroes of $B$.
\end{theorem}

\begin{proof}
%The representations are the result of the fact that the right-hand sides have the same moments as the left-hand sides.
Notice that if $|z|$ is sufficiently small, then
\[
\left\langle\frac{w+z}{w-z},\mu\right\rangle_w=1+2\sum_{n=1}^{\infty}\langle w^{-n},\mu\rangle_w z^n=F(z),
\]
which (by (\ref{carform})) proves the desired formula for $F$.  Similarly, if $|z|$ is sufficiently large, then
\[
((H-z)^{-1}e,e)_G=-\sum_{n=0}^{\infty}\frac{1}{z^{n+1}}\langle GH^ne,e\rangle_{\ell^2}=-\sum_{n=0}^{\infty}\frac{1}{z^{n+1}}\langle H^ne,e\rangle_{\ell^2}
\]
(where we used the fact that $\Delta_1=1$) while
\[
\left\langle\frac{1}{x-z},\gamma\right\rangle_x=-\sum_{n=0}^{\infty}\frac{1}{z^{n+1}}\langle x^n,\gamma\rangle_x.
\]
Notice that if we write
\[
x^n=\sum_{j=0}^na_jp_j(x),
\]
then $a_0=\langle H^ne,e\rangle$.  The fact that $\langle p_j(x),\gamma\rangle_x=0$ for $j>0$ implies $a_0=\langle x^n,\gamma\rangle_x$, which proves the desired formula for $m$.

It remains to prove formula \eqref{CN_r}.  By induction, we see that the moments of $\mu$ are real so we have
\[
F(z)=\left\langle\frac{w+z}{w-z},\mu\right\rangle_w=\left\langle\frac{\bar{w}+z}{\bar{w}-z},\mu\right\rangle_w
\] 
for all $z$ in some neighborhood of zero.  We can then mimic the proof of \cite[Theorem 13.1.2]{OPUC2} to reach the desired conclusion in some neighborhood of zero.  We know from \cite{DS17} that $F$ is holomorphic in $\dD$ away from a finite set of isolated singularities, so the relation (\ref{CN_r}) provides an analytic continuation of $m(z+z^{-1})$ to the domain of $F$ and shows that (\ref{CN_r}) holds there.  
%This also shows that both sides of (\ref{CN_r}) have the same behavior at the singularities of $F$, which proves the desired equality for all $z\in\dD$. 
Finally, \cite[Lemma 3.4]{DS17} implies that the poles of $F$ in $\dD$ coincide with the zeroes of $B$.
\end{proof}

Now, based on the presented results, we can complement the theory of the generalized Jacobi matrices with  the following extension of \cite[Theorem 13.3.2]{OPUC2}.

\begin{theorem}\label{newoprl}
If $\log\Rl F\in L^1(\dT)$ and $\{\alpha_n\}_{n=0}^{\infty}$ is a sequence of real Verblunsky coefficients obeying \eqref{IndCond}, then the polynomials $\{P_n\}_{n=0}^{\infty}$ given by \eqref{GerPol} obey
\begin{equation}\label{SzegoAsymp}
\lim_{n\to\infty}z^nP_n\left(z+\frac{1}{z}\right)=\frac{B(z)D(0)}{B(0)D(z)},\qquad\qquad |z|<1,
\end{equation}
where the convergence is uniform on compact subsets.
\end{theorem}

\begin{proof}
We begin by noticing that \cite[Theorem 3.5]{DS17} tells us that if $\log\Rl F\in L^1(\dT)$, then $\alpha_n\rightarrow0$ as $n\rightarrow\infty$.  Next, we observe that \eqref{StarRatio} and the argument from \cite[page 92]{OPUC1} together imply
\begin{equation}\label{nrat1}
\lim_{n\to\infty}\frac{\Phi_{n}(z)}{\Phi_{n}^*(z)}=0
\end{equation}
uniformly on compact subsets of $\dD\setminus\{\lambda_j\}_{j=1}^{k}$.  Combining this fact with Theorem \ref{normalszego} and \eqref{GerPol}, we deduce
\begin{equation}\label{plim}
\lim_{n\rightarrow\infty}z^nP_n\left(z+\frac{1}{z}\right)=\lim_{n\rightarrow\infty}\Phi_{2n}^*\left(1+\frac{\Phi_{2n}}{\Phi_{2n}^*}\right)=\frac{B(z)D(0)}{B(0)D(z)}
\end{equation}
uniformly on compact subsets of $\dD\setminus\{\lambda_j\}_{j=1}^{k}$.  Finally, notice that Theorem \ref{normalszego}, equation \eqref{nrat1}, and Montel's Theorem imply $\{z^nP_n(z+1/z)\}_{n\in\dN}$ is a normal family on $\dD$ and hence \eqref{plim} holds uniformly on compact subsets of $\dD$.
\end{proof}

\begin{remark}
Unlike the classical case, we see that for a finite number of values $z\in\dD$, the limit in \eqref{SzegoAsymp} is zero.  Indeed, if $z_0\in\dC\setminus\dR$ is an eigenvalue of $H$ (and hence a singularity of the $m$-function; see \cite{DD04,DD07}), then $\{p_n(z_0)\}_{n=0}^{\infty}$ is in $\ell^2(\dN_0)$, which implies $\{P_n(z_0)\}_{n=0}^{\infty}$ is in $\ell^2(\dN_0)$ (we used \cite[Theorem 13.3.1]{OPUC2} here).  Then, if $z\in\dD$ is such that $z+1/z=z_0$, then the limit in \eqref{SzegoAsymp} is zero.  From \eqref{CN_r}, we see that this value of $z$ is a pole of $F$ in $\dD$.
\end{remark}

Finally, it is worth noting that many asymptotic results from the classical setting still hold true in the non-classical setting.  % but some of them require some small modifications.  
For example, the next consequence of Szeg\H{o}'s theorem and the Geronimus relations keeps the same form as in the classical case.

\begin{theorem}
If $\log\Rl F\in L^1(\dT)$ then for the corresponding matrix $H$ we have 
\[
\lim_{k\to\infty}b_k=0, \qquad \lim_{k\to\infty}c_k=1.
\]
\end{theorem}

\section{Examples}

In this section we provide detailed calculations of two explicit examples, which demonstrate phenomena relevant to our theorems.

\subsection{Single Large Verblunsky Coefficient}

In \cite[Theorem 4.1]{DS17}, it was shown that there are collections of non-classical Verblunsky coefficients that generate polynomial sequences that are not orthogonal with respect to any non-trivial signed measure on the unit circle.  The simplest example of such a sequence is $\{2,0,0,0,\ldots\}$, for in this case
\begin{equation}\label{Ex1}
\Phi_n(z)=z^{n-1}(z-2)
\end{equation}
and hence if there were such a signed measure of orthogonality $\mu$, then
\[
\int z^nd\mu(z)=2\int z^{n-1}d\mu=\cdots=2^n\int 1 d\mu,
\]
which is impossible unless $\mu$ is the zero measure. However, after a few attempts and by taking into account that the sequence of monomials $1, z, z^2, \dots $ is orthogonal with respect to the Lebesgue measure on $\dT$, one may guess that the polynomials $\Phi_n$ given by \eqref{Ex1} are orthogonal with respect to a bilinear form  
\begin{equation}\label{GT}
(f,g)=\int_{0}^{2\pi}\frac{f(e^{i\theta})\overline{g(e^{i\theta})}}{|e^{i\theta}-2|^2}\, \frac{d\theta}{2\pi}+
Mf(2)\overline{g(1/2)}+Mf(1/2)\overline{g(2)},
\end{equation}
where $M$ is a real number to be determined later.  Indeed, it is clear that $(\Phi_0,\Phi_0)=1/3+2M$ and
\[
(\Phi_n,\Phi_m)=\delta_{n,m}
\]
as long as $n\ne m$ and $n,m\ge 1$.  It remains to see what happens when one of the indices is 0. To this end, consider $(\Phi_n,\Phi_0)$ for some natural number $n$:
\[
(\Phi_n,\Phi_0)=(\Phi_n,1)=-\frac{1}{2}\frac{1}{2^{n-1}}+M\frac{1}{2^{n-1}}\left(-\frac{3}{2}\right).
\]
Therefore, setting $M=-1/3$ leads to the form that realizes the orthogonality for the sequence $\{\Phi_n\}_{n=0}^{\infty}$.  Using terminology from \cite{GHM10}, the bilinear form \eqref{GT} is a Geronimus transformation of the Lebesgue measure on $\dT$. It is worth mentioning here that due to \cite[Proposition 3]{GHM10} the corresponding pseudo-Carth\'eodory function is a rational perturbation of the Carath\'eodory function that corresponds to the Lebesgue measure on $\dT$. A similar situation takes place in the case of polynomials orthogonal on the real line \cite{DD10}.  

Now, let us consider the OPRL corresponding to $\Phi_n$ given by \eqref{Ex1}.  In this case, it is easy to check that
\[
z^nP_n\left(z+\frac{1}{z}\right)=z^{2n}-2z^{2n-1}-2z+1.
\]
If $|z|<1$ and we take $\nri$, then this converges to $1-2z$.  Let us check that this is consistent with Theorem \ref{newoprl}.

Indeed, in this case we can use \eqref{Khrushev} to verify that
\[
\Rl F(z)=\frac{-3}{|1-2z|^2}
\]
Therefore,
\begin{align*}
D(z)&=\exp\left(\frac{1}{4\pi}\int_0^{2\pi}\frac{\eitheta+z}{\eitheta-z}\log\left(\frac{3}{|1-2\eitheta|^2}\right)d\theta\right)\\
%&=\exp\left(\frac{1}{4\pi}\int_0^{2\pi}\frac{\eitheta+z}{\eitheta-z}\log\left(\frac{3}{|1-2\eitheta|^2}\right)d\theta\right)\\
&=\exp\left(\frac{1}{4\pi}\int_0^{4\pi}\frac{\eitheta+z}{\eitheta-z}\log\left(\frac{3/4}{|\eitheta-1/2|^2}\right)d\theta\right)\\
&=\frac{\sqrt{3}}{2-z},
\end{align*}
where we used the table \cite[page 85]{OPUC1}.  Also, we notice that
\[
B(z)=\frac{1/2-z}{1-z/2}=\frac{1-2z}{2-z}
\]
so that $B(0)=1/2$.  Then
\[
\frac{B(z)D(0)}{B(0)D(z)}=\frac{(1-2z)\sqrt{3}}{\sqrt{3}}=1-2z
\]
exactly as predicted by Theorem \ref{newoprl}.

\subsection{Single Nontrivial Moment}

Consider the Verblunsky coefficient sequence given by
\[
\left\{2\sqrt{2},\frac{-1}{2\sqrt{2}},\frac{-1}{7},\frac{-1}{12\sqrt{2}},\frac{-1}{41}\ldots,\frac{-2}{(\sqrt{2}+1)^{n+1}-(\sqrt{2}-1)^{n+1}},\ldots\right\}
\]
This corresponds to the sequence of Verblunsky coefficients for the measure
\begin{align}\label{muex2}
d\mu(\theta)=\left(1-\frac{\cos(\theta)}{\sqrt{2}}\right)\frac{d\theta}{2\pi},
\end{align}
but with a $2\sqrt{2}$ appended to the beginning of the sequence.  If we let $\tilde{\Phi}_n$ and $\tilde{\Psi}_n$ respectively denote the monic degree $n$ orthogonal and second kind polynomials for the measure (\ref{muex2}), then the formulas in \cite[Section 3.4]{OPUC1} tell us that
\begin{align}\label{pnex2}
z^nP_n\left(z+\frac{1}{z}\right)=\frac{(1-2\sqrt{2}z)(\tilde{\Phi}_{2n-1}-\tilde{\Psi}_{2n-1}+\tilde{\Phi}_{2n-1}^*+\tilde{\Psi}_{2n-1}^*)}{2\left(1+\frac{2}{(\sqrt{2}+1)^{2n}-(\sqrt{2}-1)^{2n}}\right)}+\\
\nonumber +\frac{(z-2\sqrt{2})(\tilde{\Phi}_{2n-1}+\tilde{\Psi}_{2n-1}+\tilde{\Phi}_{2n-1}^*-\tilde{\Psi}_{2n-1}^*)}{2\left(1+\frac{2}{(\sqrt{2}+1)^{2n}-(\sqrt{2}-1)^{2n}}\right)}.
\end{align}
It is clear that the denominators of these fractions converge to $2$ as $\nri$.  To estimate the numerator, we need to know more about $\tilde{\Phi}_{2n-1}$ and $\tilde{\Psi}_{2n-1}$.  We have detailed information about $\tilde{\Phi}_{n}$ from the table\footnote{\label{typo}In Table (v) from \cite[page 86]{OPUC1}, we believe there is a typo in the formula for $d_n^-$.  The correct formula appears in \cite[Equation 1.6.18]{OPUC1}.} on \cite[page 86]{OPUC1}.  Formulas for $\tilde{\Psi}_n$ are more difficult to obtain.  However, we can use \cite[Propositiion 3.2.8]{OPUC1}, which tells us that\footnote{We believe  there is a typo in \cite[Equation 3.2.51]{OPUC1}.  The quantity $\overline{\Phi_n(z^{-1})}$ should be $\overline{\Phi_n(\bar{z}^{-1})}$.}
\begin{align*}
\tilde{\Psi}_n(z)&=\int_0^{2\pi}\frac{\eitheta+z}{\eitheta-z}[\tilde{\Phi}_n(\eitheta)-\tilde{\Phi}_n(z)]\left(1-\frac{\cos(\theta)}{\sqrt{2}}\right)\frac{d\theta}{2\pi}\\
\tilde{\Psi}_n^*(z)&=z^n\int_0^{2\pi}\frac{\eitheta+z}{\eitheta-z}[\overline{\tilde{\Phi}_n(\bar{z}^{-1})}-\overline{\tilde{\Phi}_n(\eitheta)}]\left(1-\frac{\cos(\theta)}{\sqrt{2}}\right)\frac{d\theta}{2\pi}
\end{align*}
Since the measure (\ref{muex2}) has all moments zero except the $0^{th}$ moment and the first moment, these integrals can be evaluated by hand.  Indeed, if
\[
\tilde{\Phi}_n(z)=\sum_{j=0}^nt_jz^j,
\]
then we calculate
\begin{align*}
\tilde{\Psi}_n(z)&=\tilde{\Phi}_n(z)+\frac{\tilde{\Phi}_n(z)-t_0}{z\sqrt{2}}-t_0+\frac{1}{2\sqrt{2}}t_1\\
\tilde{\Psi}_n^*(z)&=\tilde{\Phi}_n^*(z)\left(1-\frac{z}{\sqrt{2}}\right)-z^n\left(t_0\left(1-\frac{z}{\sqrt{2}}\right)-\frac{t_1}{2\sqrt{2}}\right)
\end{align*}
With this information, we can evaluate the numerators in (\ref{pnex2}).  The formulas in \cite[page 86]{OPUC1}\footref{typo} show that
$t_j=d_j^-/d_n$, where
\[
d_j^-=\frac{\mu_+^{j+1}-\mu_-^{j+1}}{\mu_+-\mu_-}, \quad j=0,1,\dots, n
\]
and $\mu_{\pm}=\sqrt{2}\pm1$ since we set $a=1/\sqrt{2}$.
In particular, the latter formula yields that $t_0$ and $t_1$ tend to $0$ as $n\to\infty$.  Furthermore, since the measure (\ref{muex2}) is a Szeg\H{o} measure, we know that $\tilde{\Phi}_n(z)\rightarrow0$ as $n\to\infty$ for every $z\in\dD$ and so the above calculations show that the same is true for $\tilde{\Psi}_n(z)$.  Finally, we again use the fact that (\ref{muex2}) is a Szeg\H{o} measure and the formula for the Szeg\H{o} function for this measure (see \cite[page 86]{OPUC1}) to write (when $|z|<1$)
\[
\lim_{n\to\infty}\tilde{\Phi}_n^*(z)=\frac{1}{1-(\sqrt{2}-1)z},\qquad\qquad\lim_{n\to\infty}\tilde{\Psi}_n^*(z)=\frac{1-z/\sqrt{2}}{1-(\sqrt{2}-1)z}.
\]
Putting all of this together, we find that if $|z|<1$ and we send $\nri$, then
\begin{align*}
\lim_{\nri}z^nP_n\left(z+\frac{1}{z}\right)&=\frac{(1-2\sqrt{2}z)(2-z/\sqrt{2})+(z-2\sqrt{2})(z/\sqrt{2})}{2(1-(\sqrt{2}-1)z)}\\
%&=\frac{z^3(6+5\sqrt{2})-z^2(26+14\sqrt{2})+z(8+13\sqrt{2})-4}{-4(z^2-2\sqrt{2}z+1)}\\
&=\left(\frac{6+5\sqrt{2}}{-4}\right)\frac{(z-2\sqrt{2})\left(z-\frac{2\sqrt{2}-1}{7}\right)}{z-(\sqrt{2}+1)}
\end{align*}
One can check using the table on \cite[page 86]{OPUC1} with $a=1/\sqrt{2}$ that
\[
\Rl F(z)=\frac{7(1-|z-2\sqrt{2}|^2)}{|(z-2\sqrt{2})(1-z(2\sqrt{2}+1))|^2}
\]
Therefore,
\begin{align*}
D(z)&=\exp\left(\frac{1}{4\pi}\int_0^{2\pi}\frac{\eitheta+z}{\eitheta-z}\log\left(\frac{-7(1-|\eitheta-2\sqrt{2}|^2)}{|(\eitheta-2\sqrt{2})(1-\eitheta(2\sqrt{2}+1))|^2}\right)d\theta\right)\\
&=\sqrt{7}\exp\left(\frac{1}{4\pi}\int_0^{2\pi}\frac{\eitheta+z}{\eitheta-z}\log\left(|\eitheta-2\sqrt{2}|^2-1\right)d\theta\right)\\
&\qquad\qquad\times\frac{1}{\sqrt{8}}\exp\left(\frac{1}{4\pi}\int_0^{2\pi}\frac{\eitheta+z}{\eitheta-z}\log\left(\frac{1}{|\eitheta-\frac{1}{2\sqrt{2}}|^2}\right)d\theta\right)\\
&\qquad\qquad\qquad\times\frac{1}{2\sqrt{2}+1}\exp\left(\frac{1}{4\pi}\int_0^{2\pi}\log\left(\frac{1}{|\eitheta-1/(2\sqrt{2}+1))|^2}\right)d\theta\right)\\
&=\frac{\sqrt{28}(1-z(\sqrt{2}-1))}{\sqrt{2-\sqrt{2}}(2\sqrt{2}-z)(2\sqrt{2}+1-z)}
\end{align*}
In particular
\[
D(0)=\frac{\sqrt{7}}{(2\sqrt{2}+1)\sqrt{4-2\sqrt{2}}}
\]
Finally, notice that the formulas in \cite[Section 3.4]{OPUC1} imply
\begin{align*}
\Phi_n^*(z)&=\frac{1}{2}\left((1-2\sqrt{2}z)(\tilde{\Phi}_{n-1}^*(z)+\tilde{\Psi}_{n-1}^*(z))+(z-2\sqrt{2})(\tilde{\Phi}_{n-1}^*(z)-\tilde{\Psi}_{n-1}^*(z))\right)\\
&\rightarrow\frac{1}{2}\left(\frac{1-z2\sqrt{2}+1-z/\sqrt{2}}{1-(\sqrt{2}-1)z}+\frac{z-2\sqrt{2}-1+z/\sqrt{2}}{1-(\sqrt{2}-1)z}\right)
\end{align*}
as $\nri$, where we used \cite[Corollary 3.2.5]{OPUC1}.  Simplifying this expression, we see that the only zero of $\Phi_n^*(z)$ in $\overline{\dD}$ for large $n$ approaches $(2\sqrt{2}+1)^{-1}$ as $\nri$.  Therefore,
\[
\lim_{\nri}B_n(z)=B(z)=\frac{\frac{1}{2\sqrt{2}+1}-z}{1-\frac{z}{2\sqrt{2}+1}}
\]
We then find
\begin{align*}
\frac{B(z)D(0)}{B(0)D(z)}&=\frac{\sqrt{7}(2\sqrt{2}+1)\left(\frac{1}{2\sqrt{2}+1}-z\right)\sqrt{2-\sqrt{2}}(2\sqrt{2}-z)(2\sqrt{2}+1-z)}{(2\sqrt{2}+1)\sqrt{4-2\sqrt{2}}\sqrt{28}(1-z(\sqrt{2}-1))\left(1-\frac{z}{2\sqrt{2}+1}\right)}\\
&=-\frac{(6+5\sqrt{2})\left(z-\frac{2\sqrt{2}-1}{7}\right)(z-2\sqrt{2})}{4(z-(\sqrt{2}+1))}
\end{align*}
exactly as predicted by Theorem \ref{newoprl}.

\end{document}